\documentclass[12pt,leqno]{amsart}
\usepackage{amssymb,amsthm,amsmath,latexsym,wasysym}
\usepackage{soul}
\usepackage[all]{xy}

\newtheorem{theorem}{\sc Theorem}[section]
\newtheorem{thm}[theorem]{\sc Theorem}
\newtheorem{lem}[theorem]{\sc Lemma}

\newtheorem{cor}[theorem]{\sc Corollary}

\newtheorem{rem}[theorem]{\sc Remark}

\newtheorem*{con1'}{Conjecture 1'}

\title[Weak Commutativity]{On the exponent of the Weak commutativity group $\chi(G)$}
\author[Bastos]{R. Bastos}
\address{ Departamento de Matem\'atica, Universidade de Bras\'ilia,
Brasilia-DF, 70910-900 Brazil }
\email{(Bastos) bastos@mat.unb.br; (de Melo) emerson@mat.unb.br}
\author[de Melo]{E. de Melo}
\author[de Oliveira]{R. de Oliveira}
\address{ Instituto de Matem\'atica e Estat\'istica, Universidade Federal de Goi\'as,
Goi\^ania-GO, 74690-900 Brazil }
\email{(de Oliveira) ricardo@ufg.br}
\subjclass[2010]{20D10, 20D15, 20E06, 20E34, 20J99}
\keywords{Solvable groups; $p$-groups; weak commutativity}

\begin{document}
\maketitle
\begin{abstract}
The weak commutativity group $\chi(G)$ is generated by two isomorphic groups $G$ and $G^{\varphi }$ subject to the relations $[g,g^{\varphi}]=1$ for all $g \in G$. The group $\chi(G)$ is an extension of $D(G) = [G,G^{\varphi}]$ by $G \times G$. We prove that if $G$ is a finite solvable group of derived length $d$, then $\exp(D(G))$ divides $\exp(G)^{d}$ if $|G|$ is odd and $\exp(D(G))$ divides $2^{d-1}\cdot \exp(G)^{d}$ if $|G|$ is even. Further, if $p$ is a prime and $G$ is a $p$-group of class $p-1$, then $\exp(D(G))$ divides $\exp(G)$. Moreover, if $G$ is a finite $p$-group of class $c\geq 2$, then $\exp(D(G))$ divides $\exp(G)^{\lceil \log_{p-1}(c+1)\rceil}$ ($p\geq 3$) and $\exp(D(G))$ divides $2^{\lfloor \log_2(c)\rfloor} \cdot \exp(G)^{\lfloor \log_2(c)\rfloor+1}$ ($p=2$).
\end{abstract}

\maketitle

\section{Introduction}

Let  $G^{\varphi}$ be
a copy of the group $G$, isomorphic via $\varphi : G \rightarrow
G^{\varphi}$, given by $g \mapsto g^{\varphi}$. The following group construction was introduced
and analyzed in \cite{Sidki} $$ \chi(G) = \langle G \cup G^{\varphi} \mid [g,g^{\varphi}]=1, \forall g \in G \rangle.$$ The weak commutativity group $\chi (G)$ maps onto $G$ by $g\mapsto g$, $g^{\varphi }\mapsto g$
with kernel $L(G)=\left\langle g^{-1}g^{\varphi } \mid g\;\in G\right\rangle $ and
maps onto $G\times G$ by $g\mapsto \left( g,1\right) ,g^{\varphi
}\mapsto \left( 1,g\right) $ with kernel $D(G)= [G,G^{\varphi}] = \langle [g,h^{\varphi}] \mid g,h \in G \rangle$. It is an
important fact that $L(G)$ and $D(G)$ commute. Define $T(G)$ to be the
subgroup of $G\times G\times G$ generated by $\{(g,g,1),(1,g,g)\mid g\in G\}$. Then $\chi (G)$ maps onto $T(G)$ by $g\mapsto \left( g,g,1\right)$, $g^{\varphi }\mapsto \left( 1,g,g\right) $, with kernel $W(G)=L(G)\cap D(G) \leq Z(L(G)D(G))$ and $\chi(G)/W(G)$ is isomorphic to a subgroup of $G \times G \times G$.  Another normal subgroup of $\chi (G)$ is $R(G)={%
[G,L(G),G^{\varphi }]} \leq W(G)$. By \cite[Lemma 2.2]{Roc82} the following sequence $$ 1 \to  R(G) \to W(G) \to M(G) \to 1$$ is exact, where $M(G)$ is the Schur multiplier of $G$ (cf. \cite{Miller}). For a fuller treatment we refer the reader to \cite{Roc82,Sidki}. See also \cite{GRS,LO}. 

In \cite{Roc82}, N.\,R. Rocco describes some bounds on the order and nilpotency class of $\chi(G)$, when $G$ is a finite $p$-group, $p$ odd. Recently, in \cite{BdMGN}, the authors prove that if $p$ is odd and $G$ is a powerful $p$-group, then $D(G)$ and all $\gamma_k(\chi(G))$ are powerfully embedded in $\chi(G)$. Moreover, if $G$ is a powerful $3$-group, then $\exp(\chi(G))$ divides $3\cdot \exp(G)$; and if $p\geq 5$, then $\exp(\chi(G))=\exp(G)$. 

In \cite{Moravec}, Moravec showed that if $G$ is a finite solvable group, then $\exp(M(G))$ divides $\exp(G)
^{2d-1}$. Later, in \cite{Sambonet}, Sambonet improves the previous bound proving that if $G$ is a finite solvable $p$-group, then $\exp(M(G))$ divides $2^{d-1}\cdot \exp(G)^{d}$ if $p=2$ and $\exp(M(G))$ divides $\exp(G)^{d}$ if $p>2$. Now, we extends this bound for the exponent $\exp(D(G))$ and  $\exp(\chi(G))$.        

\begin{thm} \label{thm:solvable}
Let $G$ be a finite solvable group of derived length $d$. 
\begin{itemize}
    \item[i)]  If $|G|$ is odd, then $\exp(D(G))$ divides  $\exp(G)^{d}$ and $\exp(\chi(G))$ divides  $\exp(G)^{d+1}$.
    \item[ii)] If $|G|$ is even, then $\exp(D(G))$ divides  $2^{d-1} \cdot \exp(G)^{d}$ and $\exp(\chi(G))$ divides  $2^{d-1} \cdot \exp(G)^{d+1}$. 
\end{itemize}
\end{thm}
 
In \cite{Ellis}, Ellis proved that if $G$ is a $p$-group of class $c\geq 2$, then $\exp(M(G))$ divides $\exp(G)^{c-1}$. In \cite{Moravec}, Moravec showed that if $G$ is a $p$-group of class $c\geq 2$, then $\exp(M(G))$ divides $\exp(G)^{2\cdot \lfloor \log_2(c)\rfloor}$. Later, in \cite{Sambonet17}, Sambonet proved that if $G$ is a $p$-group of class $c\geq 2$, then $\exp(M(G))$ divides $\exp(G)^{\lfloor \log_{p-1}(c)\rfloor+1}$ if $p>2$ and $\exp(M(G))$ divides $2^{\lfloor \log_2(c)\rfloor} \cdot \exp(G)^{\lfloor \log_2(c)\rfloor+1}$. We obtain the following related results.

\begin{thm}\label{nilG}
Let $G$ be a $p$-group of class at most $p-1$. Then $D(G)$ has class at most $p-1$, $\exp(D(G))$ divides $\exp(G)$ and $\exp(\chi(G))$ divides $\exp(G)^2$.   
\end{thm}

\begin{thm}\label{thm:nilpotent}
Let $G$ be a $p$-group of class $c$. 
\begin{itemize}
    \item[i)]  If $p$ is odd, then $\exp(D(G))$ divides $(\exp{(G)})^{n}$ and $\exp(\chi(G))$ divides $(\exp{(G)})^{n+1}$, where $n=\lceil \log_{p-1}(c+1)\rceil$.
    \item[ii)] If $p=2$, then $\exp(D(G))$ divides $2^{m} \cdot \exp(G)^{m+1}$ and $\exp(\chi(G))$ divides $2^{m} \cdot \exp(G)^{m+2}$, where $m = \lfloor \log_2(c)\rfloor$. 
\end{itemize}
\end{thm}

We will denote by $\tau(G)$ the quotient $\chi(G)/R(G)$. In particular, the subgroup $D(G)_{\tau(G)} = [G,G^{\varphi}]_{\tau(G)}$ is isomorphic to the non-abelian exterior square $G \wedge G$ and $W(G)_{\tau(G)}$ is isomorphic to the Schur multiplier $M(G)$ (see \cite{Roc82,NR2} for more details). As the non-abelian exterior square $G \wedge G$ is a section of $D(G)$, it seems natural to obtain bounds on the exponent $\exp(D(G))$ using known bounds for $\exp(G \wedge G)$ and $\exp(R(G))$. Unfortunately, this approach does not seem easy to implement in practice, see Remark \ref{rem.R(G)} for more details. 

\section{Preliminary Results}

The following theorem is known as P. Hall's collection formula (see \cite[Theorem 2.6]{G} for more details).

\begin{thm} \label{thm.Hall} 
Let $G$ be a $p$-group and $x, y$ elements of $G$. Then for any $k\geq 0$ we have  
\[
(xy)^{p^k}\equiv x^{p^k}y^{p^k} \pmod{\gamma_{2}(L)^{p^k}\gamma_{p}(L)^{p^{k-1}}\gamma_{p^2}(L)^{p^{k-2}}\gamma_{p^3}(L)^{p^{k-3}}\cdots \gamma_{p^k}(L)},
\]
where $L=\langle x,y\rangle$. We also have 
\[
[x,y]^{p^k}\equiv [x^{p^k}, y] \pmod {\gamma_{2}(L)^{p^k}\gamma_{p}(L)^{p^{k-1}}\gamma_{p^2}(L)^{p^{k-2}}\ldots \gamma_{p^k}(L)},
\]
where $L=\langle x,[x,y]\rangle$.
\end{thm}

The next lemma is well-known (see for instance \cite[Lemma 2.2.2]{gorenstein}).

\begin{lem}\label{arbG}
Let $x$ and $y$ be arbitrary elements of a group and suppose that $z=[x,y]$ commutes with both $x$ and $y$. Then $(yx)^i=y^ix^i z^{\frac{i(i-1)}{2}}$.
\end{lem}


The remainder of this section is dedicated to describe some structural results of $\chi(G)$ which will be useful in the proof of the main results. The following basic properties are consequences of 
the defining relations of $\chi(G)$ and the commutator rules (see \cite[Proposition 4.1.13]{Sidki} and \cite[Lemma 2.1]{Roc82} for more details). 

\begin{lem} 
\label{basic.chi}
The following relations hold in $\chi(G)$, for all 
$x, y,y_i,z,z_i \in G$.
\begin{itemize}
\item[$(a)$] $[x,y^{\varphi}] = [x^{\varphi},y]$;
\item[$(b)$] $[x,y^{\varphi}]^{z^{\varphi}} = [x,y^{\varphi}]^z$;
\item[$(c)$] $[x,y^{\varphi}]^{\omega(z_1^{\varepsilon_1},\ldots,z_n^{\varepsilon_n})} = [x,y^{\varphi}]^{\omega(z_1,\ldots,z_n)}$ for $\varepsilon_i \in \{1,\varphi\}$ and any word $\omega(z_1,\ldots,z_n) \in G$; 
\item[$(d)$] $[x^{\varphi},y,x]=[x,y,x^{\varphi}]$;
\item[$(e)$] $[x^{\varphi},y_1,\ldots,y_n,x] = [x,y_1,\ldots,y_n,x^{\varphi}]$.
\end{itemize}
\end{lem}

\begin{rem} \label{rem.R(G)} 
According to \cite[Section 2]{NR2}, we deduce that the following sequence $$ 1 \to R(G) \to  D(G) \to G \wedge G \to 1$$ is exact, where $R(G)=[G,L(G),G^{\varphi}]$ is an abelian group. In particular, $\exp(D(G))$ divides $\exp(G \wedge G) \cdot \exp(R(G))$. Consequently, it seems natural to obtain bounds on the exponent $\exp(D(G))$ using known bounds for $\exp(G \wedge G)$ and $\exp(R(G))$. Unfortunately, this approach does not seem easy to implement in practice, because the known description to $R(G)$ is unsuitable to give bounds for $\exp(R(G))$. Actually, few information can be said to respect of $R(G)$ for an arbitrary group $G$. See \cite[Section 4]{LO} for a description of $R(G)$ for some specific cases of $G$.   
\end{rem}

Let $N$ be a normal subgroup of a finite group $G$. We set $\overline{G}$ for the quotient group $G/N$ and the canonical epimorphism $\pi: G \to \overline{G}$ gives rise to an epimorphism $\widetilde{\pi}: \chi(G) \to \chi(\overline{G})$ such that $g \mapsto \overline{g}$, $g^{\varphi} \mapsto \overline{g^{\varphi}}$, where $\overline{G^{\varphi}} = G^{\varphi}/N^{\varphi}$ is identified with $\overline{G}^{\varphi}$. We will denoted by $\tilde{\pi}_{D}$ the restriction of $\tilde{\pi}$ to $D(G)$.

\begin{lem}(Sidki, \cite[Propositions 4.1.12~(i) and 4.1.13~(ii)--(iii)]{Sidki}) \label{lem.Sidki}
With the above notation we have 

\begin{itemize}
\item[$(a)$] $\ker(\tilde{\pi}_D) = [N,G^{\varphi}] \unlhd \chi(G)$;
\item[$(b)$] The following sequence: $$ 1 \to  [N,G^{\varphi}] \to D(G) \to [\overline{G},\overline{G}^{\varphi}] \to 1 $$ is exact, where $[\overline{G},\overline{G}^{\varphi}] = D(\overline{G}) \leq \chi(\overline{G})$. 
\end{itemize}
\end{lem}

In \cite{Sidki}, Sidki proved that if $G$ is a finite $p$-group, then so is $\chi(G)$. Later, in \cite{GRS}, Gupta, Rocco and Sidki proved that if $G$ is a $2$-generator nilpotent group of class $c$, then $\chi(G)$ is nilpotent of class at most $c+1$. The next two results are consequences of these above results. 

\begin{lem}\label{chi_sub_H}
Let $H$ be a $2$-generator $p$-subgroup of a finite group $G$. If $H$ has class $c$, then $\langle H,H^{\varphi} \rangle \leq \chi(G)$ is a $p$-group of class at most $c+1$. 
\end{lem}

\begin{proof}
Since $\langle H,H^{\varphi} \rangle \leq \chi(G)$ is an epimorphic image of $\chi(H)$, it suffices to prove that $\chi(H)$ is a finite $p$-group of class at most $c+1$. 

As $H$ is a finite $p$-group, then $\chi(H)$ is also a finite $p$-group (Sidki's criterion, \cite[Theorem C (i)]{Sidki}). Now, it remains to conclude that $\langle H,H^{\varphi}\rangle$ has nilpotency class at most $c+1$. Now, the result follows from \cite[Theorem 3.1]{GRS}.
\end{proof}

\begin{lem}\label{xy}
Let $G$ be a $p$-group and $x,g \in G$, such that $H=\langle x, [x,g]\rangle$ is nilpotent of class $c$. Then $\tilde{H}= \langle x, [x,g^\varphi]\rangle \leq \chi(G)$ is nilpotent of class at most $c+1$.
\begin{proof}
Set $\tilde{X} = \{x,[x,g^{\varphi}]\}$. By Lemma \ref{basic.chi}~(c), $[x^{\varphi},g,x] = [x,g^{\varphi},x] = [x,g,x^{\varphi}]$. From this we can write $\gamma_{c+2}(\tilde{H})$ as follows 
\begin{eqnarray*}
\gamma_{c+2}(\tilde{H}) & \leq  & \langle  [\alpha_1,\ldots,\alpha_{c+2}] \mid \alpha_i \in \tilde{X}]\} \rangle^{\chi(G)} \\ 
& = & \langle  [[x,g^{\varphi}],x,\alpha_3,\ldots,\alpha_{c+2}], [x,[x,g^{\varphi}],\alpha_3,\ldots,\alpha_{c+2}] \mid \alpha_i \in \tilde{X} \rangle^{\chi(G)}.
\end{eqnarray*}

Now, it is sufficient to prove that any previous basic commutator is trivial. First assume that $\alpha_{c+2} = [x,g^{\varphi}]$. According to \cite[Proposition 4.1.7]{Sidki}, we deduce that $W(G) \leq Z(D(G))$. Since $\left[[x,g^\varphi],x, \alpha_3, \cdots, \alpha_{c+1} \right]$, $\left[x,[x,g^\varphi],  \alpha_3, \cdots, \alpha_{c+1} \right] \in W(G) \leq Z(D(G))$, for any $\alpha_3,\ldots,\alpha_{c+1}\in \tilde{X}$, it follows that 
$$
\begin{array}{rcl}
  \left[[x,g^\varphi],x, \alpha_3, \cdots, \alpha_{c+1}, [x,g^{\varphi}] \right]   & = & \left[[x,g^\varphi],x,  \alpha_3, \cdots, \alpha_{c+1},  [x,g^{\varphi}] \right] = 1
\end{array}
$$
and
$$
\begin{array}{rcl}
  \left[x,[x,g^\varphi], \alpha_3, \cdots, \alpha_{c+1}, [x,g^{\varphi}] \right]   & = & \left[x,[x,g^\varphi],  \alpha_3, \cdots, \alpha_{c+1},  [x,g^{\varphi}] \right] = 1.
\end{array}
$$

Now, it remains to consider $\alpha_{c+2} = x$. By Lemma \ref{basic.chi}~(d), we can rewrite the previous basic commutators of $\gamma_{c+2}(\tilde{H})$ as follows: 

$$[[x,g^{\varphi}],x,\alpha_3,\ldots,\alpha_{c+2}] = [[x,g^{\varphi}],x,\beta_3,\ldots,\beta_{c+2}]$$ 
and 
$$[x,[x,g^{\varphi}],\alpha_3,\ldots,\alpha_{c+2}] = [x,[x,g^{\varphi}],\beta_3,\ldots,\beta_{c+2}]$$ where $\beta_i\in \{x,[x,g]\}$. By Lemma \ref{basic.chi} (a) and (e), $[x,[x,g^{\varphi}]] = [x^{\varphi},[x,g]]$. It follows that
$$
\begin{array}{rcl}
  \left[[x,g^\varphi],x, \beta_3, \cdots, \beta_{c+1}, x \right]   & = & \left[[x^{\varphi},g],x,\beta_3, \cdots, \beta_{c+1},x \right] \\
  & = & \left[[x,g],x,\beta_3, \cdots, \beta_{c+1},x^\varphi \right] \in \left[ \gamma_{c+1}(H), x^\varphi \right]=1.
\end{array}
$$

$$
\begin{array}{rcl}
  \left[x,[x,g^\varphi],\beta_3, \cdots, \beta_{c+1}, x \right]   & = & \left[x^{\varphi},[x,g],\beta_3, \cdots, \beta_{c+1},x \right] \\
  & = & \left[x,[x,g],\beta_3, \cdots, \beta_{c+1},x^\varphi \right] \in \left[ \gamma_{c+1}(H), x^\varphi \right]=1.
\end{array}
$$
because $H$ has class $c$. The proof is complete. 
\end{proof}
\end{lem}

\section{Proof of Theorem \ref{thm:solvable}}

The first lemma deduce the structure (and exponent) of the kernel $\ker(\tilde{\pi}_{D})$, where $\pi: G \to G/A$ the canonical homomorphism and $A$ is an abelian normal subgroup of $G$ (cf. Lemma \ref{lem.Sidki}, above). This lemma will be essential in the proof of Theorem \ref{thm:solvable}.    

\begin{lem}\label{lem.kernel}
Let $G$ be a finite group and $A$ an abelian normal subgroup of $G$. Then $[A,G^{\varphi}]$ is nilpotent of class at most $2$. Moreover,  
\begin{itemize}
    \item[$(a)$] $\exp([A,G^{\varphi}])$ divides $\exp(A)$, if  $\exp(A)$ is odd;
    \item[$(b)$] $\exp([A,G^{\varphi}])$ divides $2 \cdot \exp(A)$, if $\exp(A)$ is even. 
\end{itemize}
\end{lem}
\begin{proof}
Note that $[[a_1,g_1^{\varphi}],[a_2,g_2^{\varphi}]] \in Z(D(G))$, $a_i \in A$, $g_i \in G$. Indeed, by Lemma \ref{basic.chi} (c),  

$$[g,h^{\varphi}]^{[[a_1,g_1^{\varphi}],[a_2,g_2^{\varphi}]]} = [g,h^{\varphi}]^{[[a_1,g_1],[a_1,g_1]]} = [g,h^{\varphi}],$$ for any $g,h \in G$. From this we deduce that $[A,G^{\varphi}]$ is nilpotent of class at most $2$. 

By Lemma \ref{basic.chi} (a) and (d) we have that $[a,g^{\varphi},a]^a=[a,g,a^{\varphi}]^a=[a,g^{\varphi},a]$ since $[a,g,a]=1$ and $[a^{\varphi},a]=1$. Hence $[a,g^{\varphi},a]$ commute with $a$. It is also clear that by Lemma \ref{basic.chi}~(c) we have that $[a,g^{\varphi},a]$ commute with $[a,g^{\varphi}]$.  Now, using Lemma \ref{arbG} with $x=[a,g^{\varphi}]$ and $y=a$ we obtain 
$$(a[a,g^{\varphi}])^i = a^i [a,g^{\varphi}]^i [a,g^{\varphi},a]^{\frac{i(i-1)}{2}}$$

$$(a^{g^{\varphi}})^i=(a^i)^{g^{\varphi}} = a^i [a,g^{\varphi}]^i [a,g^{\varphi},a]^{\frac{i(i-1)}{2}}.$$ Then

$$a^{-i}(a^{i})^{g^{\varphi}}=[a^i,g^{\varphi}] = [a,g^{\varphi}]^i[a,g^{\varphi},a]^{\frac{i\cdot (i-1)}{2}},i \geq 2.$$

Since $[a,g^{\varphi},a]$ commute with $a$ we have that $[a,g^{\varphi},a^i]=[a,g^{\varphi},a]^i$. Then $|[a,g^{\varphi}]|$ divides $|a|$. It allow us to conclude that $\exp([A,G^{\varphi}])$ divides $\exp(A)$ if $|A|$ is odd, since $[A,G^{\varphi}]$ is nilpotent of class at most 2. On the other hand, if $\exp(A)=e$ is even, then $\frac{2e\cdot (2e-1)}{2}$ is divided by $e$. Then in this case $\exp([A,G^{\varphi}])$ divides $2\cdot \exp(A)$. The proof is complete.
\end{proof}

We are now in a position to prove Theorem \ref{thm:solvable}. 

\begin{proof}[Proof of Theorem \ref{thm:solvable}]
Since $\chi(G)$ is an extension of $D(G)$ by $G \times G$, it suffices to prove that $\exp(D(G))$ divides $\exp(G)^d$, if $\exp(G)$ is odd and $\exp(D(G))$ divides $2^{d-1}\cdot \exp(G)^d$ if $\exp(G)$ is even. 

We argue by induction on the derived length of $G$,  $d$. If $G$ is abelian, the result follows from Lemma \ref{lem.kernel}. Suppose by induction hypothesis that if $K$ is a finite solvable group with derived length $s< d$, then $\exp(D(K))$ divides $\exp(K)^{s-1}$ if $|K|$ is odd or $2^{s-2}\cdot \exp(K)^{s-1}$ if $|K|$ is even. 

Consider the last nontrivial term of the derived series of $G$ and set $\overline{G} = G/G^{(d-1)}$. According to Lemma \ref{lem.Sidki}~(b), we obtain the following exact sequence
$$ 
[G^{(d-1)},G^{\varphi}]  \hookrightarrow D(G)\twoheadrightarrow D(\overline{G}).  
$$ 
Consequently, $$\exp(D(G)) \ \ \text{divides} \ \ \exp \left( D(\overline{G}) \right) \cdot \exp([G^{(d-1)},G^{\varphi}]).$$ 
By Lemma \ref{lem.kernel}, $\exp([G^{(d-1)},G^{\varphi}])$ divides $\exp(G^{(d-1)})$ if $\exp(G^{(d-1)})$ is odd and divides $2\cdot \exp(G^{(d-1)})$ if $\exp(G^{(d-1)})$ is even. On the other hand, the derived length of $\overline{G}$ is $d-1$ and then we use induction hypothesis to establish the formula.
\end{proof}

As an immediate consequence of the above result we obtain Sambonet's theorem \cite[Theorem A]{Sambonet} concerning the exponent of the Schur multiplier of $p$-groups of derived length $d$. More precisely, we deduce the following extension.  

\begin{cor}
Let $G$ be a finite $p$-group of derived length $d$. 
\begin{itemize}
    \item[i)] Then $\exp(M(G))$ and $\exp(G \wedge G)$ divide  $\exp(G)^{d}$, $p\geq 3$.
    \item[ii)] Then $\exp(M(G))$ and $\exp(G \wedge G)$ divide $2^{d-1} \cdot \exp(G)^{d}$, $p=2$.
\end{itemize}
\end{cor}

\section{Proof of Theorems \ref{nilG} and \ref{thm:nilpotent}}

A finite $p$-group $G$ is called regular if $x^py^p \equiv~(xy)^p \mod  H^p$, for every $x,y \in G$ and $H = H(x,y) = \langle x,y\rangle'$. It is well-known that if $G$ is a regular $p$-group and $G$ is generated by a set $X$, then $\exp(G) = \max\{|x| \mid x \in X\}$ (see \cite[Corollary 2.11]{G}). We are now in a position to prove Theorem \ref{nilG}.

\begin{proof}[Proof of Theorem \ref{nilG}]
Since $\chi(G)$ is an extension of $D(G)$ by $G \times G$, it is sufficient to show that $\exp(D(G))$ divides $\exp(G)$. 

Let $\omega \in \gamma_{p-1}([G,G^{\varphi}])$. By Lemma \ref{basic.chi} (c) we have that $[h,g^{\varphi}]^{\omega}=[h,g^{\varphi}]$ since $\gamma_{p-1}([G,G])=1$. Thus, $\gamma_{p-1}([G,G^{\varphi}]) \leq Z(D(G))$, in particular $D(G)$ has class at most $p-1$. Hence the subgroup $D(G)$ is a regular $p$-group and it is sufficient to proof that each generator has order at most $\exp(G)$ (cf. \cite[Corollary 2.11]{G}).

Let $[x,g^{\varphi}]$ where $x,g\in G$. First, we will show that if $x^{p^e}$, then $[x,g^{\varphi}]^{p^e} \in [G',G^{\varphi}]^{p^e}$. Let $H$ be the subgroup of $G$ generated by $\langle x, g \rangle$. It is clear that $\gamma_p(H)=1$. Hence by Lemma \ref{chi_sub_H} the subgroup $\tilde{H}$ of $\chi(G)$ generated by  $\langle x, g, x^{\varphi},g^{\varphi} \rangle$ has class at most $p$. By Theorem \ref{thm.Hall} we have that \[
[x,g^{\varphi}]^{p^e}\equiv [x^{p^e}, g^{\varphi}] \pmod {\gamma_{2}(L)^{p^e}\gamma_{p}(L)^{p^{e-1}}\gamma_{p^2}(L)^{p^{e-2}}\ldots \gamma_{p^e}(L)},
\]
where $L=\langle x,[x,g^{\varphi}]\rangle$. Since $\gamma_{p}(L)\leq \gamma_{p+1}(\tilde{H})=1$ we obtain that
$$
[x,g^{\varphi}]^{p^e} \in  \gamma_{2}(L)^{p^e}.
$$ On the other hand, since $[x,g^{\varphi},x]=[x,g,x^{\varphi}]\in [G',G^{\varphi}]$ we conclude that $[x,g^{\varphi}]^{p^e} \in \gamma_{2}(L)^{p^e} \leq [G',G^{\varphi}]^{p^e}$. Thus $[G,G^{\varphi}]^{p^e} \leq [G',G^{\varphi}]^{p^e}$. Now repeating the process with $a\in G^{'}$ and $b^{\varphi}\in G^{\varphi}$ we obtain that $[G,G^{\varphi}]^{p^e} \leq [G',G^{\varphi}]^{p^e} \leq [\gamma_3(G),G^{\varphi}]^{p^e}$. Now it is clear that using this process we obtain that $[G,G^{\varphi}]^{p^e}=1$, since $\gamma_p(G)=1$.
\end{proof}

\begin{lem}\label{lemP1}
Let $G$ be a $p$-group of class $c$. Let $A=\gamma_{\left\lceil\frac{c+1}{p-1}\right\rceil}(G)$. Then $[A,G^{\varphi}]$ has class at most $p-1$ and $\exp([A,G^{\varphi}])$ divides $\exp(A)$.  
\end{lem}
\begin{proof}
Let $\omega \in \gamma_{p-1}([A,G^{\varphi}])$. By Lemma \ref{basic.chi} (c) we have that $[h,g^{\varphi}]^{\omega}=[h,g^{\varphi}]$ since $\gamma_{p-1}([A,G])=1$. Thus, $\gamma_{p-1}([A,G^{\varphi}]) \leq Z(D(G))$, in particular $\gamma_{p}([A,G^{\varphi}])=1$.

Let $[x,g^{\varphi}]$ where $x\in A$ and $g\in G$. We will show that if $x^{p^e}=1$, then $[x,g^{\varphi}]^{p^e}=1$. By Theorem \ref{thm.Hall} we have that \[
[x,g^{\varphi}]^{p^e}\equiv [x^{p^e}, g^{\varphi}] \pmod {\gamma_{2}(L)^{p^e}\gamma_{p}(L)^{p^{e-1}}\gamma_{p^2}(L)^{p^{e-2}}\ldots \gamma_{p^e}(L)},
\]
where $L=\langle x,[x,g^{\varphi}]\rangle$.  By Lemma \ref{xy} we have that $\gamma_p(L)=1$ since $\langle x,[x,g]\rangle$ is nilpotent of class at most $p-2$. Thus, $
[x,g^{\varphi}]^{p^e} \in  \gamma_{2}(L)^{p^e}$. On the other hand, $[x,g^{\varphi},x]=[x,g,x^{\varphi}]\in [A,A^{\varphi}]$ and by Lemma \ref{basic.chi} (c) $[A,A^{\varphi}]$ is normalized by $L$. Hence we conclude that $\gamma_2(L) \leq [A,A^{\varphi}] $ and then $\gamma_2(L)^{p^e}=1$ by Theorem \ref{nilG}.
\end{proof}

Let us now prove Theorem~\ref{thm:nilpotent}. 

\begin{proof}[Proof of Theorem~\ref{thm:nilpotent}]
Since $\chi(G)$ is an extension of $D(G)$ by $G \times G$, it suffices to consider $\exp(D(G))$. \\ 

\noindent i) Assume that $p \geq 3$. We argue by induction on $c$. If $c \leq p-1$, then the result follows from Theorem \ref{nilG}. Now, suppose by induction hypothesis that if $M$ is a $p$-group of class $r$ with $2 \leq  r< c$, then $\exp(D(K))$ divides $\exp(K)^{\lceil \log_{p-1}(r+1)\rceil}$. 

Let $A=\gamma_{\left\lceil\frac{c+1}{p-1}\right\rceil}(G)$. Set $\overline{G} = G/A$. By Lemma \ref{lem.Sidki}~(b) we have the following induced exact sequence
$$ 
[A,G^{\varphi}] \hookrightarrow D(G)\twoheadrightarrow D\left( \overline{G} \right).
$$
By Lemma \ref{lemP1}, $\exp([A,G^{\varphi}])$ divides $\exp(G)$. By induction, the last term of the exact sequence $D(\overline{G})$ has exponent dividing $\exp(G)^{m}$  where $m=\left\lceil \log_{p-1} \left\lceil\frac{c+1}{p-1}\right\rceil\right\rceil=\left\lceil \log_{p-1}(\frac{c+1}{p-1})\right\rceil=\left\lceil\log_{p-1}(c+1)\right\rceil-~1$. Combining these bounds, one obtains $\exp(D(G))$ divides $\exp(G)^{\lceil \log_p(c+1)\rceil}$. \\ 

\noindent ii) Assume that $p = 2$. Set $m= \lfloor \log_2(c)\rfloor$ and $d$ the derived length of $G$. According to Theorem \ref{thm:solvable}~ii) and the formula $d \leq \lfloor \log_2(c)\rfloor +1$ (cf. \cite[Theorem 5.1.12]{Rob}), we deduce that $\exp(D(G))$ divides $2^{d-1} \cdot \exp(G)^d$ and so, $ \exp(D(G))$ divides  $2^{m}\cdot \exp(G)^{m+1}$, which completes the proof. 
\end{proof}

As an immediate consequence of the above result we obtain Sambonet's theorem \cite[Theorem 1.1]{Sambonet17} concerning the exponent of the Schur multiplier of $p$-groups of class $c$. More precisely, we conclude the following extension.  

\begin{cor}
Let $G$ be a finite $p$-group of class $c$. Then 
\begin{itemize}
    \item[i)] $\exp(M(G))$ and $\exp(G \wedge G)$ divide  $\exp(G)^{\lceil \log_{p-1}(c+1)\rceil}$, $p\geq 3$.
    \item[ii)] $\exp(M(G))$ and $\exp(G \wedge G)$ divide $2^{\lfloor \log_2(c)\rfloor} \cdot \exp(G)^{\lfloor \log_2(c)\rfloor+1}$, $p=2$.
\end{itemize}
\end{cor}

\section*{Acknowledgements}

This work was partially supported by DPI/UnB and FAPDF (Brazil).

\end{document}